\documentclass{article}
\usepackage{amsmath,amsthm,amssymb,hyperref,graphicx,float} 
\usepackage{color}	
\usepackage{epsfig}
\newcommand{\nc}{\newcommand}

\newtheorem{thm}{Theorem}[section]

\newtheorem{prop}[thm]{Proposition}
\newtheorem{lemma}[thm]{Lemma}

\newtheorem{corollary}[thm]{Corollary}

\newtheorem{definition}[thm]{Definition}

\newenvironment{defin}{\begin{definition} \rm}{\end{definition}}

\newenvironment{lem}{\begin{lemma}\rm }{\end{lemma}}

\nc{\Ext}{\operatorname{Ext}}
\nc{\FS}{\operatorname{FS}}
\nc{\NS}{\operatorname{NS}}
\nc{\Amp}{\operatorname{Amp}}
\nc{\Pic}{\operatorname{Pic}}
\nc{\Kom}{\operatorname{Kom}}
\nc{\DGrB}{\operatorname{DGrB}}
\nc{\antidiag}{\operatorname{antidiag}}
\nc{\diag}{\operatorname{diag}}
\nc{\mo}{\operatorname{mod}}
\nc{\Gr}{\operatorname{Gr}}
\nc{\Rep}{\operatorname{Rep}}
\nc{\Perf}{\operatorname{Perf}}
\nc{\Hom}{\operatorname{Hom}}
\nc{\Sym}{\operatorname{Sym}}
\nc{\RHom}{R\operatorname{Hom}}
\nc{\cRHom}{\operatorname{\mathcal{R}\mathcal{H}om}}
\nc{\cHom}{\operatorname{\mathcal{H}om}}
\nc{\End}{\operatorname{End}}
\nc{\Coh}{\operatorname{Coh}}
\nc{\Aut}{\operatorname{Aut}}
\nc{\Td}{\operatorname{Td}}
\nc{\Coker}{\operatornamoe{Coker}}
\nc{\coker}{\operatorname{coker}}
\nc{\colim}{\operatorname{colim}}
\nc{\Ker}{\operatorname{Ker}}
\nc{\img}{\operatorname{Im}}
\nc{\D}{\operatorname{D}}
\nc{\ch}{\operatorname{ch}}
\nc{\Stab}{\operatorname{Stab}}
\nc{\SL}{\operatorname{SL}}
\nc{\rk}{\operatorname{rk}}
\nc{\GL}{\operatorname{GL}}
\nc{\Log}{\mathop{\mathrm{Log}}}
\nc{\abs}[1]{\lvert#1\rvert}
\nc{\Cone}{\operatorname{Cone}}
\nc{\id}{\operatorname{id}}
\nc{\Li}{\operatorname{Li}}
\nc{\fmod}{\operatorname{fmod}}
\nc{\SdR}{\operatorname{SdR}}

\newcommand{\Db}{{\mathrm D}^{b}}

\nc{\cA}{{\mathcal A}}
\nc{\cB}{{\mathcal B}}
\nc{\cC}{{\mathcal C}}
\nc{\cD}{{\mathcal D}}
\nc{\cE}{{\mathcal E}}
\nc{\cF}{{\mathcal F}}
\nc{\cG}{{\mathcal G}}
\nc{\cH}{{\mathcal H}}
\nc{\cI}{{\mathcal I}}
\nc{\cJ}{{\mathcal J}}
\nc{\cK}{{\mathcal K}}
\nc{\cL}{{\mathcal L}}
\nc{\cM}{{\mathcal M}}
\nc{\cN}{{\mathcal N}}
\nc{\cO}{{\mathcal O}}
\nc{\cP}{{\mathcal P}}
\nc{\cQ}{{\mathcal Q}}
\nc{\cR}{{\mathcal R}}
\nc{\cS}{{\mathcal S}}
\nc{\cT}{{\mathcal T}}
\nc{\cU}{{\mathcal U}}
\nc{\cV}{{\mathcal V}}
\nc{\cW}{{\mathcal W}}
\nc{\cX}{{\mathcal X}}
\nc{\cY}{{\mathcal Y}}
\nc{\cZ}{{\mathcal Z}}
 
\nc{\bA}{{\mathbb A}}
\nc{\bB}{{\mathbb B}}
\nc{\bC}{{\mathbb C}}
\nc{\bD}{{\mathbb D}}
\nc{\bE}{{\mathbb E}}
\nc{\bF}{{\mathbb F}}
\nc{\bG}{{\mathbb G}}
\nc{\bH}{{\mathbb H}}
\nc{\bI}{{\mathbb I}}
\nc{\bJ}{{\mathbb J}}
\nc{\bK}{{\mathbb K}}
\nc{\bL}{{\mathbb L}}
\nc{\bM}{{\mathbb M}}
\nc{\bN}{{\mathbb N}}
\nc{\bO}{{\mathbb O}}
\nc{\bP}{{\mathbb P}}
\nc{\bQ}{{\mathbb Q}}
\nc{\bR}{{\mathbb R}}
\nc{\bS}{{\mathbb S}}
\nc{\bT}{{\mathbb T}}
\nc{\bU}{{\mathbb U}}
\nc{\bV}{{\mathbb V}}
\nc{\bW}{{\mathbb W}}
\nc{\bX}{{\mathbb X}}
\nc{\bY}{{\mathbb Y}}
\nc{\bZ}{{\mathbb Z}}
    
\begin{document}
\title{Mirror stability conditions and SYZ conjecture for Fermat
    polynomials} \author{So Okada\footnote{Supported by JSPS
    Grant-in-Aid and Global Center of Excellence Program at Kyoto
    University, Email: okada@kurims.kyoto-u.ac.jp, Address: Research
    Institute for Mathematical Sciences, Kyoto University, 606-8502,
    Japan. }}
\maketitle

\begin{abstract}
  Calabi-Yau Fermat varieties are obtained from moduli spaces of
  Lagrangian connect sums of graded Lagrangian vanishing cycles on
  stability conditions on Fukaya-Seidel categories. These graded
  Lagrangian vanishing cycles are stable representations of quivers 
  on their {\it mirror stability conditions}.
\end{abstract}

\section{Introduction}
For a projective variety $X$, Strominger-Yau-Zaslow conjectured that
we can obtain the space $X$ as a moduli space of special Lagrangians
with $U(1)$ connections when we have a mirror pair of $X$ and argued
locally and physically.  The author recommends \cite{Aur} for a recent
illustration on the subject and \cite{HKKPTVVZ} for a general
reference.
 
In this paper, we study the conjecture categorically in the framework
of homological mirror symmetry of Kontsevich \cite{Kon95} and
stability conditions of Bridgeland \cite{Bri07}. The latter notion,
which was inspired by Douglas' so-called $\Pi$-stabilities in
superstring theory \cite{Dou02,Dou01}, categorizes King's
$\theta$-stabilities \cite{Kin} (\cite[Section
5.3]{BriTol},\cite[Section 7.3.3]{Asp},\cite{Ber}).  In particular,
for stability conditions of quivers without relations, when deformed
on the stability manifold if necessarily, stable representations solve
self-dual equations twisted by characters of general linear groups
\cite[Section 6]{Kin}.

On derived categories defined in terms of graded Lagrangian vanishing
cycles over morsifications of Fermat polynomials
$X_{n}:=x_{1}^{n}+\ldots +x_{n}^{n}:\bC^{n}\to \bC$
\cite{Sei01,AurKatOrl08}, we put pairs of stability conditions which
are named {\it mirrors} in Definition \ref{def:mirror} and defined on
quivers with commuting relations with the following properties (see
Theorem \ref{thm:syz}).  For one in each pair, with framing in
Definitions \ref{def:rest} and \ref{def:rest_rep} in terms of a finite
group of tensor products of Auslander-Reiten transformations, we have
a moduli space of stable Lagrangian connect sums of graded Lagrangian
vanishing cycles; this moduli space via the Serre-de Rham functor in
Definition \ref{def:twisted} gives the Calabi-Yau Fermat variety in
$\bP^{-1}$ defined by the zero locus of $X_{n}$.  For the other one in
the pair, graded Lagrangian vanishing cycles are stable.

Stable objects above are stable representations of quivers even when
we forget relations. In particular, graded Lagrangian vanishing cycles
trivially satisfy self-dual equations.  In this paper, connections are
over transverse intersections of graded Lagrangian vanishing cycles
and connect trivial vector bundles on graded Lagrangian vanishing
cycles; so, they are maps between vector spaces and consistency with
grading make them representations of quivers with relations.
  
\section{Mirror stability conditions}\label{sec:mir}
Let us recall the notion of stability conditions.  For a triangulated
category $\cT$, each stability condition is determined by a bounded
$t$-structure and a stability function $Z$, which gives central
charges of elements of the Grothendieck group of $\cT$, on the heart
of the $t$-structure.

We go on with explicit examples for our later use.  Let $A_{n-1}$
denote the A-type Dynkin quiver of vertices $0,\ldots, n-1$ and arrows
$0\to \ldots \to n-1$ and let $A_{n-1}^{\otimes n}$ denote the
$n$-fold tensor product of $A_{n-1}$ (see \cite{Les} for the general
definition of tensor products of quivers).

Each vertex $a$ of $A_{n-1}^{\otimes n}$ can be labelled by $n$-tuples
$a^{1},\ldots, a^{n}$ for $0\leq a^{i} \leq n-1$. We have an arrow
$a\to b$ when $b^{i}-a^{i}=0$ for all $i$ but at most one $j$ such
that $b^{j}-a^{j}=1$; we let $\lambda(a,b)=j$ if we have such $j$ or
$\lambda(a,b)=0$ if $a=b$.  We have commuting relations on arrows. For
our convenience in this paper, for each vertex $a$ of
$A_{n-1}^{\otimes n}$, we call the number $\lambda(a):=\sum a^{i}$ the
{\it index} of $a$.

For example, the following figure shows commuting arrows on rectangles
which connect vertices of indices $0,1,2$ for $A_{4}^{\otimes 5}$.
\begin{figure}[H] 
  \begin{center} 
    \input{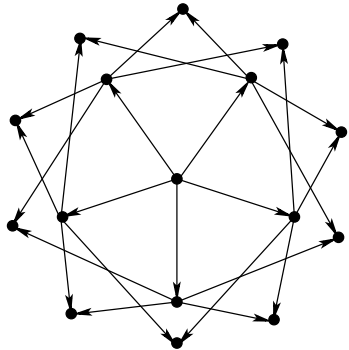}\label{fig:partial quintic quiver}
    \caption{A part of $A_{4}^{\otimes 5}$}
  \end{center}
\end{figure}    
   
For the triangulated category $\Db(\mo A_{n-1}^{\otimes n})$ and the
heart $\mo A_{n-1}^{\otimes n}$ of the bounded $t$-structure, we have
a stability function which maps simple representations into the
upper-half plane of the complex plane.  We mean by $Z_{n}$ a stability
function on the heart $\mo A_{n-1}^{\otimes n}$ with the following
conditions: slopes of simple representations, which are
one-dimensional representations over vertices, strictly decreases as
their indices increases and central charges of simple representations
of the same indices are the same.  Choices of such stability functions
gives the open submanifold of the stability manifold of $\Db(\mo
A_{n-1}^{\otimes n})$.

The following figure shows one of such stability functions on $\mo
A_{4}^{\otimes 5}$. Arrows indicate increases of coordinates of the
complex plane.  Central charges of simple representations put
dots. For example, central charges of simple representations of
vertices with the index $1$ put the second left-most dot.
 \begin{figure}[H] 
 \begin{center} 
  \input{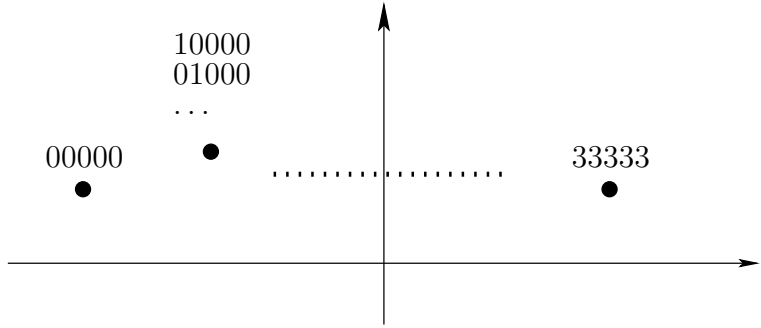}\label{fig:central_charge_p}
   \caption{A stability function $Z_{5}$ on $\mo A_{4}^{\otimes 5}$}
 \end{center}
\end{figure}    

For the stability function above, we have the stability condition of
$\Db(\mo A_{4}^{\otimes 5})$ with the stability function
$-\bar{Z_{5}}$ on the same heart $\mo A_{4}^{\otimes 5}$.
\begin{figure}[H]
   \begin{center}
   \input{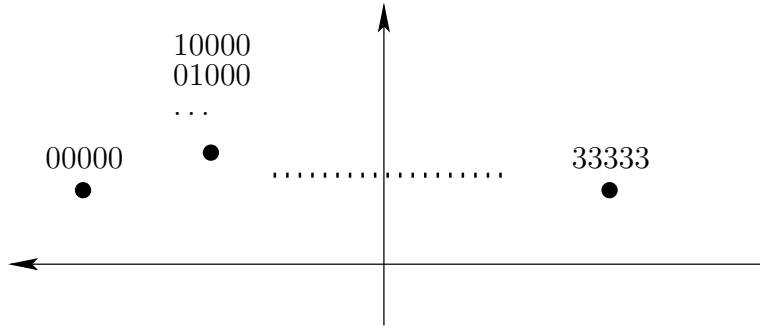}\label{fig:central_charge_q}
   \caption{The stability function $-\bar{Z_{5}}$ on $\mo A_{4}^{\otimes 5}$}
 \end{center}
\end{figure}
   
For our convenience in this paper, we call the stability condition of
$\Db(\mo A_{n-1}^{\otimes n})$ determined by the stability function
$-\bar{Z_{n}}$ on the heart $\mo A_{n-1}^{\otimes n}$ is {\it mirror}
to the stability condition determined by the stability function
$Z_{n}$ on the same heart $\mo A_{n-1}^{\otimes n}$.  Between these
mirrors, we have wall-crossing paths on the stability manifold of
$\Db(\mo A_{n-1}^{\otimes n})$.  We put the formal definition as
follows.

\begin{defin}\label{def:mirror}
  Let $\cA$ be the heart of a bounded $t$-structure of a triangulated
  category $\cT$ such that $\cA$ is isomorphic to the category of
  representations of a quiver with or without relations.  Let $Z$ be a
  stability function on the heart such that central charges of simple
  representations are in the upper-half plane of the complex plane. We
  call the stability function $-\bar{Z}$ on $\cA$ {\it mirror} to the
  stability function $Z$, and the stability condition with the
  stability function $Z$ on $\cA$ {\it mirror} to the other.
\end{defin}

For a nontrivial directed acyclic graph such as of $A_{n-1}^{\otimes
  n}$, we have a stability condition with stable objects being simple
representations, and we have the mirror stability condition with a
representation of a nontrivial support being a stable object.

We can state notions in Definition \ref{def:mirror} in a more general
way; however, without a finiteness property on the heart such as
\cite[Proposition 2.3]{Bri07}, even if we assume the existence of a
stability condition, the existence of the mirror stability condition
is obscure. Examples of mirror stability conditions in terms of
spherical functors or wall-crossings can be found in
\cite{BayMac,Bri09,Bri05,IshUedUeh,KajSaiTak,KonSoi,Mac,
  Ohk,Oka06a,Oka06b}.

\section{Recap of homological mirror symmetry}
In this paper, derived equivalence $\cong$ means that on both sides of
the equivalence, we have compact generating objects with
$A_{\infty}$-isomorphic $A_{\infty}$ enhancements; this is an example
of so-called derived Morita equivalence.

In \cite{Oka09}, we have seen that $X_{n}$ is self-dual up to the
equivariance with respect to the group $H_{n}$ which consists of
$n$-tuples of $n$-th roots of the unity modulo diagonals and which
acts on coordinates.  For example, we have $\Db_{H_{n}}(\Coh
X_{n})\cong \FS(X_{n})$ for the $H_{n}$-equivariant bounded derived
category of coherent sheaves of the Calabi-Yau Fermat variety and the
Fukaya-Seidel category of $X_{n}$ (see \cite{FutUed} for a different
account).

We have a morsification of $X_{n}$ and an $A_{\infty}$ algebra of
graded Lagrangian vanishing cycles in terms of Lagrangian intersection
theories such that the $A_{\infty}$ algebra is $A_{\infty}$-isomorphic
to the Yoneda ($\Ext$) algebra of simple representations of
$A_{n-1}^{\otimes n}$.  For simple representations $S_{a}$ and $S_{b}$
of vertices $a$ and $b$ of $A_{n-1}^{\otimes n}$ with an arrow $a\to
b$, we have an one-dimensional $\Ext^{1}(S_{a},S_{b})$ morphism and
these morphisms along arrows anti-commute.

For each $0\leq i \leq n-1$, let $O_{n}^{i}$ denote the object
$\Omega_{\bP^{n-1}}^{n-1-i}(n-1-i)[i]$ restricted on the Calabi-Yau
Fermat variety in $\bP^{n-1}$. Bases of
$\Ext^{1}(O_{n}^{i},O_{n}^{i+1})$ can be given by morphisms $d
x_{n,i}^{j}$ for $1\leq j\leq n$ such that we have anti-commuting
relations $d x_{n,i}^{j} d x_{n,i-1}^{j'}= - d x_{n,i}^{j'} d
x_{n,i-1}^{j}$ for $1\leq j, j'\leq n$.  Putting (weighted) copies of
$d x_{n,i}^{j}$ on arrows $a\to b$ such that $\lambda(a,b)=j$ and
$\lambda(a)=i$, it is easy to check that $H_{n}$-equivariant objects
of $O_{n}^{i}$ give also the Yoneda algebra.

For the dual statement of $\Db_{H_{n}}\Coh (X_{n})\cong \FS(X_{n})$,
enhancing the argument above, we have defined the $\hat{H}_{n}$
quotient $\FS^{\hat{H}_{n}}(X_{n})$ of $\FS(X_{n})$ as the perfect
derived category of the dg $\hat{H}_{n}$-orbit category (see
\cite{Kel05}) of a dg algebra $\cA$ which is $A_{\infty}$ isomorphic
to the $A_{\infty}$ algebra of objects in $\hat{H}_{n}$-orbits of
graded Lagrangian vanishing cycles so that the dg $\hat{H}_{n}$-orbit
category is a dg enhancement of the Yoneda algebra of objects
$O_{n}^{i}$.  This is a formal application of Pontryagin duality on
categories of dg categories, and $\Db(\Coh X_{n})\cong
\FS^{\hat{H}_{n}}(X_{n})$.  As for the existence of such $\cA$, we can
take the dg algebra of $\hat{H}_{n}$-graded matrix factorizations of
$O_{n}^{i}$ in terms of the dg category of graded matrix
factorizations of $X_{n}$ \cite{Orl}.
 
We have $\FS(X_{n})\cong \Db(\mo A_{n-1}^{\otimes n})\cong \Db(\mo
A_{n-1})^{\otimes n}$ in terms of dg tensor products (see
\cite{Kel06}).  We have tensor products of Auslander-Reiten
transformations $\tau$ of $\Db(\mo A_{n})$ such that $\tau^{n}\cong
[2]$. Actions of the finite group consisting of $\tau^{t_1}\otimes
\cdots \otimes \tau^{t_{n}}$ such that $\sum t_{i}=0$ coincides with
those of $\hat{H}_{n}$ above and the index of each simple
representation of $A_{n-1}^{\otimes n}$ stays the same by the actions.
   
\section{Main statement}
Directly, it is highly nontrivial to discuss stability conditions and
moduli spaces on complexes of $O_{n}^{i}$ with nontrivial $A_{\infty}$
structures to take into account (see \cite{DouGovJay}), even with
motivating realizations in terms of coherent sheaves or
Lagrangians.

By taking equivariance, we have stability conditions on $\mo
A_{n-1}^{\otimes n}$, and on $\Db_{H_{n}}(\Coh X_{n})$.  Still, an
issue we face to consider stability conditions on $\Db(\Coh X_{n})$
with stability conditions on $\mo A_{n-1}^{\otimes n}$ is that such
stability conditions are not invariant under $\hat{H}_{n}$; because,
not all objects which consist of the $\hat{H}_{n}$-orbit of a simple
representation of $ A_{n-1}^{\otimes n}$ are representations of
$A_{n-1}^{\otimes n}$.  If they were invariant, then we would have
taken advantages of the paper \cite{Pol} by Polishchuk and the paper
\cite{MacMehSte} by Macr\`{i}-Mehrotra-Stellari.

We overcome the issue by taking the notion of {\it framed
  $\hat{H}_{n}$-invariance} on stability conditions in Definition
\ref{def:rest} and representations and morphisms in Definition
\ref{def:rest_rep}; instead of full products of general linear groups
over vertices, we take ones which commutes with $\hat{H}_{n}$ actions
restricted on simple representations of $A_{n-1}^{\otimes n}$.

 \begin{defin}\label{def:rest}
   We say that a stability function $Z$ on $\mo A_{n-1}^{\otimes n}$
   is {\it framed $\hat{H}_{n}$-invariant}, if central charges of
   simple representations of $ A_{n-1}^{\otimes n}$ of each index are
   the same.
 \end{defin}
 
 For example, stability functions $Z_{i}$ and their mirrors in Section
 \ref{sec:mir} are framed $\hat{H}_{n}$-invariant.  To define framed
 $\hat{H}_{n}$-invariant representations, for each representation $E$
 of $A_{n-1}^{\otimes n}$, let $E_{a,b}:E_{a}\to E_{b}$ denote
 commuting linear maps along arrows $a\to b$; in particular, $E_{a,a}$
 are identity maps on $E_{a}$.
 
 \begin{defin}\label{def:rest_rep}
   We say that a representation $E$ of $A_{n-1}^{\otimes n}$ is {\it
     framed $\hat{H}_{n}$-invariant}, if for vertices $a,b$ of
   $A_{n-1}^{\otimes n}$ with the same indices, we have {\it framing
     isomorphisms} $\phi_{a,b}:E_{a}\to E_{b}$ with the following
   conditions. For vertices $a,b,c, a',c'$ such that $a\to a'$ and
   $c\to c'$ with $\lambda(a,a')=\lambda(c,c')$ and
   $\lambda(a)=\lambda(b)=\lambda(c)$, we have
   $\phi_{a',c'}E_{a,a'}=E_{c,c'}\phi_{b,c}\phi_{a,b}$; i.e., we have
   the following commuting diagram such that squig arrows indicate
   isomorphisms and plain arrows indicate maps of a representation.
 \begin{equation*}
   \begin{matrix}
     a          & \rightsquigarrow & b & \rightsquigarrow &c\\
    \downarrow &      &   \circlearrowright     &          &\downarrow \\
     a'          & & \rightsquigarrow&     & c'
   \end{matrix}.
   \end{equation*}

   For representations $E$ and $F$ of $A_{n-1}^{\otimes n}$ with
   framed $\hat{H}_{n}$-invariance, we say that a morphism $f:E\to F$
   of $A_{n-1}^{\otimes n}$ is {\it framed $\hat{H}_{n}$-invariant},
   if for vertices $b,b'$ such that $\lambda(b)=\lambda(b')$, we have
   $\phi^{F}_{b,b'}f_{b}\phi^{E}_{b',b}=f_{b'}$.
   
   Let $\fmod A_{n-1}^{\otimes n}$ denote the {\it category of framed
     $\hat{H}_{n}$-invariant representations} which consists of framed
   $\hat{H}_{n}$-invariant representations and morphisms of
   $A_{n-1}^{\otimes n}$.
 \end{defin}

 In Definition \ref{def:rest_rep}, by putting $a=c$ and $a'=c'$, we
 see that $\phi_{b,c}\phi_{a,b}=\phi_{a,c}$, by putting $a=c$, we have
 $\phi_{b,a}\phi_{a,b}=\phi_{a,a}$, and by putting $a=b=c$, we have
 $\phi_{a,a}=1_{E_{a}}$.  We have uniqueness of framing isomorphisms
 in the following sense.

 \begin{lem}\label{lem:triv}
   Let $E$ be a framed $\hat{H}_{n}$-invariant representation of
   $A_{n-1}^{\otimes n}$ supported over vertices with indices
   $0,\ldots, n-1$. For vertices $a_{i}$ of $A_{n-1}^{\otimes n}$ such
   that $\lambda(a_{i})=i$ for $0\leq i \leq n-1$, we have a framed
   $\hat{H}_{n}$-invariant isomorphism $t^{E}:E\to E'$ such that $E'$
   is a framed $\hat{H}_{n}$-invariant representation of
   $A_{n-1}^{\otimes n}$ with trivial framing isomorphisms and
   $t^{E}_{a}=\phi_{a,a_{\lambda(a)}}$ for vertices $a$ of
   $A_{n-1}^{\otimes n}$.
 \end{lem} 
 \begin{proof}
   For vertices $b_{i}$ of $A_{n-1}^{\otimes n}$ with
   $\lambda(b_{i})=i$, we let $E'_{b_{i}}:=E_{a_{i}}$, and for arrows
   $b_{i}\to b_{i+1}$, we let
   $E'_{b_{i},b_{i+1}}:=\phi_{b_{i+1},a_{i+1}}E_{b_{i},b_{i+1}}\phi_{a_{i},b_{i}}$.
   For vertices $b_{i}',b_{i+1}'$ with
   $\lambda(b_{i},b_{i+1})=\lambda(b'_{i},b'_{i+1})$, we have
   $E'_{b_{i},b_{i+1}}=E'_{b'_{i},b'_{i+1}}$. For vertices $b'_{i}$
   with $\lambda(b_{i-1},b_{i})=\lambda(b'_{i},b_{i+1})$ and
   $\lambda(b_{i-1},b'_{i})=\lambda(b_{i},b_{i+1})$, we have
   $E'_{b_{i},b_{i+1}}E'_{b_{i-1},b_{i}}
   =E'_{b'_{i},b_{i+1}}E'_{b_{i-1},b'_{i}}$.  For arrows $b_{i}\to
   b_{i+1}$, we have $t_{b_{i+1}}E_{b_{i},b_{i+1}}= E'_{b_{i},b_{i+1}}
   t_{b_{i}}$.  For vertices $b_{i}'$ with
   $\lambda(b_{i})=\lambda(b_{i}')$, we have
   $\phi^{E'}_{b_{i},b_{i}'}t^{E}_{b_{i}}\phi_{b_{i}',b_{i}}^{E}
   =t^{E}_{b_{i}'}$.
  \end{proof}

  Let us recall the quiver $B_{n-1}$, called the Be{\u\i}linson quiver
  of $\bP^{n-2}$ \cite{Bei,Min}, which has $n$ vertices $0,\ldots,
  n-1$ and $n-1$ arrows from $i$ to $i+1$ with commuting relations.
  For vertices $i,i+1$, we label arrows by $s$ for $1 \leq s \leq n$
  so that for maps $E_{i,i+1}^{s}$ on labelled arrows
  $i\stackrel{s}{\to}i+1$, we have $E_{i,i+1}^{s} E_{i-1,i}^{s'}
  =E_{i,i+1}^{s'} E_{i-1,i}^{s}$.

 \begin{prop}\label{prop:eq}
   The full subcategory of $\fmod A_{n-1}^{\otimes n}$ consisting of
   representations supported over vertices with indices $0,
   \ldots,n-1$ is equivalent to $\mo B_{n-1}$.
 \end{prop}
 \begin{proof}
   For each representation $E$ of $B_{n-1}$, we put the representation
   $F(E)$ of $A_{n-1}^{\otimes n}$ as follows.  For vertices $a,b$ of
   $A_{n-1}^{\otimes n}$ with arrows $a\to b$, we put linear maps
   $F(E)_{a,b}:=E_{\lambda(a),\lambda(b)}^{\lambda(a,b)}$ from
   $F(E)_{a}:=E_{\lambda(a)}$ to $F(E)_{b}:=E_{\lambda(b)}$.  For each
   morphism $f:E\to E'$ of representations $E$ and $E'$ of $B_{n-1}$,
   we put $F(f):F(E)\to F(E')$ by $F(f)_{a}:=f_{\lambda(a)}$ for each
   vertex $a$.
   
   To obtain an inverse $G$ of $F$, let $a_{i}$ be vertices of
   $A_{n-1}^{\otimes n}$ such that $\lambda(a_{i})=i$ for $0\leq i
   \leq n-1$. For $E$ of $\fmod A_{n-1}^{\otimes n}$ supported over
   vertices with the indices, we put $G(E)_{i}:=E_{a_{i}}$, and for
   arrows $a\to b$ such that $\lambda(a)=\lambda(a_{i})$, we put the
   linear map $G(E)_{\lambda(a_{i}),\lambda(a_{i+1})}^{\lambda(a,b)}:=
   \phi_{b,a_{i+1}}E_{a,b}\phi_{a_{i},a}$, which is independent for
   choices of such arrows $a\to b$.  For $1\leq c,c'\leq n-1$ and
   $0\leq i-1\leq n-3$, let us note that we have a vertex
   $b_{i-1},b_{i},b'_{i}, b_{i+1}$ of $A_{n-1}^{\otimes n}$ with
   $\lambda(b_{i-1})=i-1$,
   $\lambda(b_{i-1},b_{i})=\lambda(b'_{i},b_{i+1})=c$, and
   $\lambda(b_{i-1},b'_{i})=\lambda(b_{i},b_{i+1})=c'$; so, we have
   $G(E)_{i,i+1}^{c'} G(E)_{i-1,i}^{c}= G(E)_{i,i+1}^{c}
   G(E)_{i-1,i}^{c'}$.  For a framed $\hat{H}_{n}$-invariant morphism
   $f:E\to E'$, we put $G(f):G(E)\to G(E')$ such that $G(f)_{i}=
   f_{a_{i}}$; we have $G(E')_{i,i+1}^{c}G(f)_{i}=G(f)_{i+1}
   G(E)_{i,i+1}^{c}$.

   For the functor $FG$, a framed $\hat{H}_{n}$-invariant morphism
   $f:E_{1}\to E_{2}$, and $t^{E_{1}},t^{E_{2}}$ in the notation of
   Lemma \ref{lem:triv}, we have
   $t^{E_{1}}_{a}f_{a}=\phi^{E_{1}}_{a,a_{\lambda(a)}}f_{a}=f_{a_{\lambda(a)}}
   \phi^{E_{2}}_{a, a_{\lambda(a)}} =FG(f)_{a} t^{E_{2}}_{a}$ for each
   vertex $a$ of $A_{n-1}^{\otimes n}$.  For the functor $GF$, we take
   the identity functor on $\fmod A_{n-1}^{\otimes n}$.
 \end{proof}

 Similar statements to the one in Proposition \ref{prop:eq} can be
 obtained by taking other supporting vertices.  For $n=3$ and $4$, let
 us mention that derived categories $\Db(\Coh \bP^{n-2})$, which is
 derived equivalent to $\Db(\mo B_{n-1})$, have been described in
 terms of graded Lagrangian vanishing cycles and Lagrangian
 intersection theories in \cite{AurKatOrl08}.  With Serre twists in
 use, let us define the {\it Serre-de Rham functor} $\SdR$ as follows.

 \begin{defin}\label{def:twisted}
   For each representation $E$ of $\mo B_{n-1}$, we put the chain
   $\SdR(E)$ of morphisms $\sum_{1 \leq j\leq n}E_{i,i+1}^{j} \otimes
   d x_{n,i}^{j}: E_{i}\otimes O_{n}^{i} \to E_{i+1}\otimes
   O_{n}^{i+1}$ for $0\leq i\leq n-2$. For each morphism $f:E\to F$ in
   $\mo B_{n-1}$, we put the chain map $\SdR(f):\SdR(E)\to \SdR(F)$
   such that $f_{i}\otimes \id_{O_{n}^{i}}:E_{i}\otimes O_{n}^{i} \to
   F_{i}\otimes O_{n}^{i}$ for $0\leq i \leq n-1$.
 \end{defin}

 The following is in order.

 \begin{prop}
   For each representation $E$ of $\mo B_{n-1}$, we have that
   $\SdR(E)$ is a complex and $\SdR$ is a functor from $\mo B_{n-1}$
   to $\Db(\Coh X_{n})$.
 \end{prop}
 \begin{proof}
   Each $E$ satisfies commuting relations among maps. This translates
   into zero compositions of consecutive morphisms.  Morphisms between
   objects of $\mo B_{n-1}$ become morphisms of complexes.
 \end{proof}
 
 For a given framed $\hat{H}_{n}$-invariant stability function $Z$ on
 $\mo A_{n-1}^{\otimes n}$, we define the stability function $Z'$ on
 $\mo B_{n-1}$ by putting $Z'(i):=Z(a)$ for a vertex $a$ of
 $A_{n-1}^{\otimes n}$ such that $\lambda(a)=i$. In the following, we
 write $Z'$ on $\mo B_{n-1}$ also as $Z$ on $\mo B_{n-1}$.
 
 \begin{thm}\label{thm:syz}
   For a stability function $Z_{n}$ on the category of framed
   $\hat{H}_{n}$-invariant representations of $A_{n-1}^{\otimes n}$,
   we have a moduli space of stable Lagrangian connect sums of graded
   Lagrangian vanishing cycles for a morsification of $X_{n}$ such
   that the Serre-de Rham functor localizes the moduli space into the
   Calabi-Yau Fermat variety in $\bP^{n-1}$. For the mirror of $Z_{n}$
   on the category of representations of $A_{n-1}^{n}$, each graded
   Lagrangian vanishing cycle is stable.
 \end{thm}
\begin{proof}
  For a stability function $Z_{n}$ on $\mo B_{n-1}$ and stable
  representations $E$ of $B_{n-1}$ with the dimension vector
  $(1,\ldots, 1)$, commuting relations and the indecomposable property
  of the representation give nonzero $k^{E}_{i}$ such that $k^{E}_{i}
  E_{i,i+1}^{s}=E_{i-1,i}^{s}$ for $1\leq s \leq n$ and $0 \leq i \leq
  n-1$.

  If objects $O_{n}^{i}$ were unrestricted on the Calabi-Yau Fermat
  variety, then $\SdR(E)$ would have been Koszul resolutions of
  skyscraper sheaves of points in $\bP^{n-1}$. For non-isomorphic
  representations $E$ of $\mo B_{n-1}$, objects $\SdR(E)$ of $\Db(\Coh
  X_{n})$ are non-isomorphic objects supported over distinct points of
  the Calabi-Yau Fermat variety, unless isomorphic to the zero object
  outside.

  For the mirror of $Z_{n}$ on $\mo A_{n-1}^{\otimes n}$, each stable
  object is isomorphic to a simple object of $\mo A_{n-1}^{\otimes
    n}$.
 \end{proof}

 Stable objects of Theorem \ref{thm:syz} in terms of the mirror
 $Z_{n}$ are simple representations of the quiver $A_{n-1}^{\otimes
   n}$.  We may as well take the mirror of $Z_{n}$ on the category of
 framed ones of $A_{n-1}^{\otimes n}$; in this case, we obtain
 polysimple representations of $A_{n-1}^{\otimes n}$ consisting of
 graded Lagrangian vanishing cycles.

 For projective spaces, and for Calabi-Yau hypersurfaces of
 $\bP^{n-1}$ represented as $x_{1}^{n}+\ldots + x_{n}^{n}+\psi \cdot
 x_{1}\cdots x_{n}:\bC^{n}\to \bC$ for $\psi \in \bC$, we may obtain
 similar statements to Theorem \ref{thm:syz}.
 
\section*{Acknowledgements}
The author thanks Research Institute for Mathematical Sciences of
Kyoto University and Institut des Hautes \'Etudes Scientifiques for
providing him with excellent research environment.  He thanks
Professors Carqueville, Fukaya, Hori, Kajiura, Keller, Y. Kimura,
Kontsevich, Nakajima, Ohashi, Ohkawa, Stoppa, A. Takahashi, Toda,
Usnich, Weist, and Yamazaki for their useful discussions or
stimulating talks related to this paper.

\end{document}